\documentclass[12pt]{amsart}

\usepackage{hyperref}
\usepackage{setspace}
\usepackage{latexsym}
\usepackage{amsmath}
\usepackage{amsfonts,mathrsfs}
\usepackage{amssymb}
\usepackage{amsthm}
\usepackage{verbatim}
\usepackage[all]{xy} \xyoption{arc}

\usepackage[left=3cm,top=2cm,right=3cm,bottom = 2cm]{geometry}

\usepackage{graphicx,graphicx, array, hyperref, dsfont}
\usepackage{tikz}
\usepackage{xy}
\xyoption{all}
\usepackage{xcolor}

\newtheoremstyle{def}
     {10pt}
     {10pt}
     {}
     {}
     {\rmfamily\bfseries\upshape}
     {.}
     {.5em}
     {}
 \theoremstyle{def}
 \newtheorem{definition}{Definition}[section]
\newtheorem{remark}[definition]{Remark}
\newtheorem{example}[definition]{Example}

\newtheoremstyle{theorem}
     {20pt}
     {10pt}
     {\it}
     {}
     {\rmfamily\bfseries\upshape}
     {.}
     {.5em}
     {}

\theoremstyle{theorem}
\newtheorem{theorem}[definition]{Theorem}

\usepackage{color}

\newcommand{\smalltwobytwo}[4]{
\left( \begin{smallmatrix} 
  #1 & #2\\
  #3 & #4 
\end{smallmatrix}\right)}

\newcommand{\tred}[1]{\textcolor{red}{#1}}

\usepackage{xcolor}
\hypersetup{
    colorlinks,
    linkcolor={red!50!black},
    citecolor={blue!50!black},
    urlcolor={blue!80!black}
}

\usepackage{tikz-cd}
\tikzset{>=latex}

\DeclareMathOperator{\SL}{\rm{SL}}
\newcommand{\ZZ}{\mathbb{Z}}

\newcommand{\C}{\mathbb{C}}

\newcommand{\Q}{\mathbb{Q}}

\newcommand{\Z}{\mathbb{Z}}

\newcommand{\GL}{\operatorname{GL}}

\newcommand{\F}{\mathbb{F}}

\newcommand{\Ind}{\operatorname{Ind}}
\newcommand{\sym}{\operatorname{Sym}}

\newcommand{\Gal}{\operatorname{Gal}}

\newcommand{\Ad}{\operatorname{Ad}}

\begin{document}

\title{Minimal Integral Models for principal series Weil characters }

\author{Luca Candelori, Yatin Patel}
\email{candelori@wayne.edu}
\email{yatin@wayne.edu}

\address{Department of Mathematics, Wayne State University, 656 W. Kirby,  Detroit, MI 48202}

\begin{abstract}
We prove a conjecture of Udo Riese about the minimal ring of definition for  principal series Weil characters of $\SL_2(\mathbb{F}_p)$, for $p$ an odd prime.  More precisely, we show that the $(p+1)/2$-dimensional Weil characters can be realized over the ring of integers of $\Q(\sqrt{\varepsilon p})$, where $\varepsilon =(-1)^{(p-1)/2}$, and we provide explicit integral models over these quadratic rings. We do so by studying the Galois action on the integral models of  Weil characters recently discovered by Yilong Wang. 
\end{abstract}

\maketitle

\section{Introduction}

Let $\rho: G \rightarrow \GL(V)$ be an irreducible complex representation of a finite group $G$ of exponent $e$. A famous theorem of Brauer states that there is a choice of basis for $V$ (i.e. a {\em model} for $\rho$) so that $\rho(g)$ is a matrix with entries in the cyclotomic field $\Q(\zeta_e)$, for all $g\in G$. Since the entries of the character table for $G$ are always algebraic integers, it is natural to formulate a much stronger conjecture stating that, in fact, the basis can be chosen so that the matrix entries lie in the ring of integers $\Z[\zeta_e] \subseteq \Q(\zeta_e)$ \tred{\cite{CRW},\cite{KS}}. In this case, we say that $\rho$ has an {\em integral} model over the ring $R = \Z[\zeta_e]$. Proving the existence of such integral models is a notoriously difficult problem in integral representation theory. Even when existence can be ascertained, the arguments often cannot be directly adapted to explicitly construct such integral models.

In the case $G = \SL_2(\mathbb{F}_p)$, for $p\geq 3$ a prime, Udo Riese \cite{Riese} proved the existence of integral models over $\Z[\zeta_e]$. The question remains however whether the result is best possible, or whether for this particular family of groups there are integral models defined over proper subrings $R\subset  \Z[\zeta_e]$. For each irreducible character $\chi$, the entries of the character table provide a minimal ring of definition $R_{min}(\chi)$, but it is not clear a priori whether such a {\em minimal} integral model over $R_{min}(\chi)$ actually exists. For example, the character of the Steinberg representation $\chi = St$ (the unique irreducible character of dimension $p$) has entries in $\Z$, thus  $R_{min}(St) = \Z$. It can be shown that an integral model for $St$ exists over $\Z$ \cite[Prop. 1]{Riese}, \cite{Humphreys}, thus providing a minimal integral model for $St$. Similarly, minimal integral models can easily be found for the characters belonging to the irreducible principal series of $G$.

The purpose of this article is to explicitly provide minimal integral models for the Weil characters arising from the reducible principal series of 
$\SL_2(\mathbb{F}_p)$. For each $p\geq 3$, there are precisely two non-isomorphic such irreducible characters  $\xi_1$ and $\xi_2$, each of dimension $(p+1)/2$. The entries in the character table give 
$$
R_{min}(\xi_i) =  \ZZ\left[\frac{1 + \sqrt{\varepsilon p}}{2}  \right], \quad \varepsilon = \left\{ \begin{array}{cl}
1 &\text{ if } p \equiv 1 \mod 4 \\
-1 &\text{ if } p \equiv 3 \mod 4 
\end{array} \right.
$$
for both $i=1,2$. We therefore seek explicit integral models defined over this quadratic ring $ R= R_{min}(\xi_i)$, which in all cases coincides with the ring of integers of $\Q(\sqrt{\varepsilon p})$. In \cite{Riese}, it is conjectured that such minimal integral models should always exist, and existence is proved under the restriction $p\equiv 5 \mod 8$ \cite[Prop. 4]{Riese}. The methods of \cite{Riese} are based on class-field theory and do not provide {\em explicit} integral models even under the more restrictive assumptions. In this article (Theorems \ref{thm1} and \ref{thm3})  we prove the existence of minimal integral models for any prime $p$, with no restrictions, and we provide them explicitly. For example, for the prime $p=7$ and $\xi = \xi_1$ we obtain the following model 

\begin{equation*}
\xi_1(\mathfrak{s}) =  \begin{bmatrix}
-1 & \frac{1}{2}(1-\sqrt{-7}) & 0 & 0 \\
-\frac{1}{2}(1+\sqrt{-7}) & 1 & 0 & 0 \\
\frac{1}{2}(1-\sqrt{-7}) & \frac{1}{2}(1+\sqrt{-7}) & 0 & -1 \\
1 & -1 & 1 & 0
\end{bmatrix}
\end{equation*}
and
\begin{equation*}
\xi_1(\mathfrak{t})=  \begin{bmatrix}
0 & 0 & 0 & -1 \\
1 & 0 & 0 & \frac{1}{2}(1-\sqrt{-7}) \\
0 & 1 & 0 & 1 \\
0 & 0 & 1 & \frac{1}{2}(1+\sqrt{-7})
\end{bmatrix},
\end{equation*}
where $\mathfrak{s} = \smalltwobytwo{0}{-1}{1}{0}$ and $\mathfrak{t}  = \smalltwobytwo{1}{1}{0}{1}$ are the standard choices of generators for $\SL_2(\F_p)$. Our methods are based on recent work of Yilong Wang \cite{Yilong}, who provided explicit integral models for $\xi_i$ over $\ZZ[\zeta_p]$. By extensive explicit calculations, the authors observed that in fact Wang's models are defined over the minimal rings $R_{min}(\xi_i)$. This is what we prove in this article, by studying the action of the Galois group $\Gal(\Q(\zeta_p)/\Q)$ on Wang's integral models. The key observation is the somewhat surprising fact that the action of the Galois subgroup corresponding to the unique quadratic subfield of $\Q(\zeta_p)$ on the standard (non-integral) models for the Weil characters is given by conjugation by a permutation matrix, and that the same permutation matrix gives the action of the Galois subgroup on the Vandermonde matrix associated with the Weil character. The referee pointed out to the authors that this compatibility can be also be obtained as a special case of the Galois symmetry of the modular representation of a {\em modular tensor category} \cite{Dong-Lin-Ng}, \cite{reconstruction}. Indeed it would be interesting to extend the methods of this article to provide explicit minimal integral models for the modular representation of a modular tensor category. 

There are also two irreducible {\em cuspidal} Weil characters in the character table of $\SL_2(\mathbb{F}_p)$, each of dimension $(p-1)/2$. Explicit integral models for these have been provided over $\ZZ[\zeta_p]$ by P.M. Gilmer, G. Mausbaum, and P. van Wamelen \cite{PatrickM-2004} and more recently by Shaul Zemel \cite{Shaul}, who was also motivated by Wang's work. However, minimal integral models are not known for such cuspidal Weil characters, and the methods of this article do not directly extend to the cuspidal case. Riese proves existence whenever $p\equiv 3 \mod 4$ \cite[Prop. 3]{Riese}, and provides evidence for non-existence in the case $p\equiv 1 \mod 4$, but no explicit models are known in either case. This is certainly a subject for further investigation. In addition, it would be interesting to extend these results to the case $q = p^r$ a prime power, which was the original scope of Riese's conjectures.  

We now summarize the contents of this article. In Section \ref{section:WeilRep}, we recall the definition of the Weil representation, which provides explicit models for representations of $\SL_2(\F_p)$ defined over the ring $\Z[1/p,\zeta_p]$. Since conventions about the Weil representation vary between different sources, we collect all the necessary formulas in the section. In Section \ref{section:IntegralModels}, we explain how to obtain explicit models for the Weil characters of $\SL_2(\F_p)$ from the Weil representation, and we recall Wang's construction of integral models for $\xi_1$ and $\xi_2$ \cite{Yilong}. In Section \ref{Section:proofs} we prove our main results Thms. \ref{thm1}, \ref{thm2} and \ref{thm3}, which show that Wang's integral models are in fact {\em minimal}. Finally, in Section \ref{sec:examples} we provide explicit examples of the minimal integral models for $p=7$ and $p=13$ and for both $\xi_1$ and $\xi_2$.  

We would like to thank Siu-Hung Ng, Yilong Wang and Shaul Zemel for sharing many great ideas and conversations about this problem, as well as the referee for many insightful comments. The first author is supported by the U.S. Department of Energy, Office of Science, Basic Energy Sciences, under Award Number DE-SC-SC0022134. The second author is supported by a 2021-2022 Thomas C. Rumble University Graduate Fellowship through the Department of Mathematics at Wayne State University.

\section{The Weil representation}
\label{section:WeilRep}

In this section we recall the construction of the Weil representation of $\SL_2(\F_p)$, for $p>2$ a prime, first introduced in \cite{Weil}. This construction will be used to provide explicit models for the Weil characters $\xi_1$ and $\xi_2$ over the ring $\Z[1/p,\zeta_p]$, and it serves as the point of departure for  providing integral models over $\ZZ[\zeta_p]$ \cite{Yilong}. There are many formulas and conventions about the Weil representation in the literature, and some references are hard to locate or translate. We thus find it necessary to state explicitly all the facts, conventions, and formulas that we will be using in this article, while providing references for the proofs.  

Let $p>2$ be a prime, let $C_p = (\Z/p\Z, +)$ be the cyclic group of order $p$ and let $K_p:= C_p \times C_p$.  Let $\nu_p = \{ \zeta \in \C : \zeta^p = 1 \}$  be the group of $p$-th roots of unity in $\C$ and let $\zeta_p := e^{2\pi i/p}$ be the principal $p$-th root of unity. Let $Q(x):C_p \rightarrow \Q/\Z$ be a quadratic form on $C_p$ and let $B(x,y) = Q(x+y) - Q(x) - Q(y)$ be its associated bilinear form.

\begin{definition}
The \em{Heisenberg group} $H_Q$ is the set of elements $(\lambda,x,y) \in \nu_p\times K_p$, together with multiplication given by 
$$
(\lambda_1, x_1, y_1)(\lambda_2, x_2, y_2) =  (\lambda_1 \lambda_2 \,e^{2\pi i B(x_1,y_2)}, x_1+x_2, y_1+y_2).
$$
\end{definition}

Note that $\nu_p = Z(H_p)$ and $H_Q$ is a (non-trivial) central extension: $$ 1 \to \nu_p \to H_Q \to K_p \to 0.$$

The generators of $H_Q$ are $(\zeta_p, 0, 0)$, $(1,1,0)$, and $(1,0,1)$. The subgroup of $K_p$ that is generated by $(x,0)$ where $x\in\F_p$ is the called the {\em modulation subgroup}. Let $A = \{ (\lambda, x, 0), \lambda \in \nu_p,x \in C_p \}.$ Then $A$ is an index $p$  subgroup of $H_Q$, projecting onto the modulation subgroup.  Since any isomorphism $C_p \cong \nu_p$ corresponds to a choice of a primitive $p$-th root of unity, we choose the one that sends  $1 \mapsto \zeta_p$.  With this choice, given characters $\phi_1$ and $\phi_2$ of $C_p$, we define a one-dimensional character $\psi:A\rightarrow \C^{\times}$ by  
$$
\psi(\lambda, x, 0) = \phi_1(x)\phi_2(\lambda).
$$ 
We induce the representation of the subgroup $A$ to obtain an irreducible $p$-dimensional representation $\sigma := \Ind_A^{H_p}(\psi)$. The underlying vector space for this representation is 
$$
V = \left \lbrace \text{functions: } C_p \to \C \right \rbrace
$$
which has a canonical basis of delta functions $\{ \delta_0, \cdots, \delta_{p-1} \}$, defined by $\delta_j(k) = \delta_{jk}$ (Kronecker's $\delta$-function). With  respect to this basis, we obtain an explicit model for the irreducible representation   
$$ \sigma: H_Q \to \GL(V) \simeq \GL_p(\C),$$  
such that an element $(\zeta,0,0) \in \nu_p \subset H_Q$ acts by scalar multiplication via $\phi_2(\zeta)$.  In addition, the generator $(1,1,0)$ acts via the diagonal matrix with entries 
\begin{equation}
\sigma(1,1,0)_{jj} =  \phi_1(1)\phi_2(e^{2\pi iB(1,j)}), \nonumber
\end{equation}
while the generator $(1,0,1)$ acts by a permutation matrix:
\begin{equation}
\begin{split}
\sigma(1,0,1)& =  \begin{bmatrix}
0 &0 & 0& \cdots & 0  &1 \\
 1  &0 & 0& \cdots&0 & 0\\
0  &1 & 0& \cdots &0 & 0 \\
\vdots & \vdots & \vdots &\vdots &\vdots  &\vdots \\
0&0& 0 & \cdots & 1 & 0
\end{bmatrix}.
\end{split} \nonumber
\end{equation}
The ``natural character" of $\nu_p$ is defined to be the inclusion of $\nu_p$ into $\C$. In our conventions, this corresponds to $\phi_2$ being the identity on $\nu$: $\phi_2(\zeta_p^j)=\zeta_p^j$. For each prime $p$, there exists a unique (up to isomorphism) irreducible  representation of $H_Q$ such that $\nu_p$ acts by its natural character. This is the Mackey-Stone-von Neumann Theorem \cite{Gar15}: 
\begin{theorem}[Mackey-Stone-von Neumann]
For fixed non-trivial (unitary) central character, up to isomorphism
there is a unique irreducible (unitary) representation of the Heisenberg group with that central character. Further, any (unitary) representation with that central character is a multiple of that irreducible.
\end{theorem}

The Theorem is usually stated for more general Heiseberg groups, including infinite ones, arising as central extensions of symplectic vector spaces. Note however  that for the finite Heisenberg group $H_Q$ this result can be easily proved by computing the character table.

\begin{remark}
The character table of $H_Q$ can be computed by inducing characters of $A$ as above. It follows that every irreducible character of $H_Q$ has a minimal integral model defined over $\ZZ[\zeta_p]$, verifying the stronger form of Brauer's theorem stated in the Introduction. Since $H_Q$ is solvable of exponent $p$, this also follows from the more general result of \cite{CRW}. 
\end{remark} 

In order to construct the Weil representation of $\SL_2(\F_p)$, we need to consider automorphisms of $H_Q$, as follows. 

\begin{definition} The {\em Clifford group} $\mathscr{C}$ is the group of automorphisms of the Heisenberg group that fix its center $\nu_p$. An automorphism $\psi \in \mathscr{C}$ is {\em symmetric} if $\psi(\lambda, -x,-y)  = \psi(\lambda,x,y)$. Let $\mathscr{C}^{\sym}$ be the subgroup of $\mathscr{C}$ consisting of all symmetric automorphisms that fix $\nu_p$.
\end{definition}

Let $\psi \in \mathscr{C}^{\sym}(H_Q)$ and let $\sigma^\psi := \sigma \circ \psi$. Since $\psi$ fixes the center, the representation $\sigma^\psi$ is irreducible, and $\nu_p$ acts by its natural character, so that $\sigma \cong \sigma^\psi$ by the Mackey-Stone-von Neumann Theorem.  Therefore there exists a linear transformation $M(\psi) \in \GL(V)$ such that $ M(\psi) \sigma = \sigma^\psi M(\psi)$, that is, making the following diagram    
$$\begin{tikzcd}
H_Q \arrow[r, "\sigma"] \arrow[d,  swap ,"\psi"]
& \GL(V) \arrow[d,  "\Ad(M(\psi))"] \\
H_Q \arrow[r, "\sigma^{\psi}"] & \GL(V)
\end{tikzcd}$$
commutative. The linear map $M(\psi)$ is only defined up to scalars, but it can be shown \cite{Weil} that such scalars can be chosen for each $\psi$ so that the function $\psi \mapsto M(\psi)$ gives a group {\em homomorphism} $\mathscr{C}^{\sym}\rightarrow \GL(V)$. Since $p>2$ is an odd prime, there is an isomorphism $\mathscr{C}^{\sym}(H_Q) \simeq \SL_2(\F_p)$ (see \cite{FLR}) and we obtain the following:
\begin{definition} The \textit{Weil representation} 
$$
\rho_Q: \SL_2(\F_p) \longrightarrow \GL(V),
$$
is the representation of $\SL_2(\F_p)$ obtained by composing the isomorphism $\mathscr{C}^{\sym}(H_Q) \simeq \SL_2(\F_p)$ with the homomorphism $\psi \mapsto M(\psi)$.
\end{definition}
We now write an explicit model for the Weil representation using the canonical basis of delta functions for $V$. The generator $\mathfrak{s} \in \SL_2(\F_p)$ maps isomorphically to the symmetric automorphism $\psi_{\mathfrak{s}} \in \mathscr{C}^{\sym}$ given by
$$
\psi_{\mathfrak{s}}(\lambda, x, y)  = (\lambda e^{-2\pi iB(x,y)}, -y, x) .
$$
The corresponding linear transformation $M(\psi_s)$ is a discrete Fourier transform (DFT) given by 

\begin{equation}
\label{eqn:Weil-S_mat}
 \rho_Q({\mathfrak{s}}) = M(\psi_{\mathfrak{s}}) = \dfrac{1}{\sqrt{\varepsilon p}}\begin{bmatrix}
1 & 1 & 1 & \cdots & 1\\
1 & e^{-2\pi i B(1,1)} & e^{-2\pi i B(1,2)} & \cdots & e^{-2\pi i B(1,p-1)} \\
1 & e^{-2\pi i B(2,1)} & e^{-2\pi i B(2,2)} & \cdots & e^{-2\pi i B(2,p-1)} \\
\vdots & \vdots & \vdots & \cdots & \vdots \\
1 & e^{-2\pi i B(p-1,1)} & e^{-2\pi i B(p-1,2)} & \cdots & e^{-2\pi i B(p-1,p-1)} 
\end{bmatrix}.
\end{equation}

The other generator ${\mathfrak{t}} \in \SL_2(\F_p) $ maps to the symmetric automorphism $\psi_{\mathfrak{t}} \in \mathscr{C}^{\sym}$ given by 
$$\psi_{\mathfrak{t}}(\lambda, x, y) = (\lambda e^{2\pi i Q(y)}, x+y,y)$$
and  
\begin{equation}
\label{eqn:Weil-T_mat}
\rho_Q({\mathfrak{t}}) = M(\psi_{\mathfrak{t}}) = \begin{bmatrix}
1 & 0 & 0 & \cdots & 0 & 0 \\
0 & e^{2\pi i Q(1)} & 0 & \cdots &0 & 0 \\
0 & 0 & e^{2\pi i Q(2)} & \cdots& 0 & 0 \\
\vdots & \vdots & \vdots & \ddots & \vdots & \vdots \\
0 & 0 & 0 & \cdots & 0 & e^{2\pi i Q(p-1)}
\end{bmatrix}.
\end{equation}

It is easily checked that the identities $\rho_Q({\mathfrak{s}})^2 = (\rho_Q({\mathfrak{s}})\rho_Q({\mathfrak{t}}))^3$ and $\rho_Q({\mathfrak{s}})^4 = \mathds{1}$ are satisfied, so that $\rho_Q$ is indeed a representation of $\SL_2(\F_p)$ dimension $p$.

The isomorphism class of each Weil representation only depends on the equivalence class of the quadratic form $Q$. Note that there are two nonequivalent quadratic forms on $C_p$: 

\begin{equation}
\label{eqn:Q1Q2}
Q_1(x)=x^2/p \quad \text{    and    } \quad Q_2(x)=cx^2/p
\end{equation}

\noindent where $c$ is a quadratic non-residue modulo $p$.  Therefore we obtain two non-isomorphic Weil representations 
\begin{equation}
\label{eqn:WeilReps}
\rho_1, \rho_2 : \SL_2(\F_p) \longrightarrow \GL(V), 
\end{equation}
along with their explicit (non-integral) models defined over the ring $\Z[1/p,\zeta_p]$.


\section{Integral models for Weil characters}
\label{section:IntegralModels}

The Weil representations $\rho_1, \rho_2$ defined by \eqref{eqn:WeilReps} decompose into irreducible representations $\rho_i = \xi_i\oplus \pi_i$, $i=1,2$. The characters $\xi_1,\xi_2$ are the two {\em Weil characters} of $\SL_2(\F_p)$ of dimension $(p+1)/2$; they belong to the {\em principal series} of $\SL_2(\F_p)$. The other two Weil characters $\pi_1, \pi_2$ are of dimension $(p-1)/2$ and they belong to the {\em cusipidal series}.  

We now construct an explicit integral model over $\Z[\zeta_p]$ for the principal series Weil characters $\xi_1,\xi_2$, following \cite{Yilong}. Let $V^+ \subseteq V = \left \lbrace \text{functions: } C_p \to \C \right \rbrace$ be the subspace of {\em even} functions, satisfying $f(-x) = f(x)$ for all $x\in C_p$. This subspace is an invariant subspace for both $\rho_1$ and $\rho_2$, and the principal series Weil characters  $\xi_1,\xi_2$ are the restrictions of $\rho_1,\rho_2$ to $V^+$. To write down an explicit model for these characters, consider the basis for $V^+$ given by the even delta functions: 
$$b^+_0:= \delta_0, \ b^+_1:=\delta_1 + \delta_{p-1}, \ \cdots, \ \ b^+_{(p-1)/2} := \delta_{(p-1)/2} + \delta_{(p+1)/2}.$$
Then the explicit model for the Weil representation given by \eqref{eqn:Weil-S_mat} and \eqref{eqn:Weil-T_mat} can be used to give an explicit model for $\xi_1$ and $\xi_2$ over the ring $\Z[1/p,\zeta_p]$. Let  $S$ and $T$ denote the matrices of the Weil representations $\rho(\mathfrak{s})$ and $\rho(\mathfrak{t})$ respectively, restricted to the basis of even functions. Then 
\begin{equation}
\label{eqn:xi-S-matrix}
\begin{split}
S_{jk} &=\left \lbrace
\begin{array}{lr} \dfrac{\sqrt{\varepsilon p}}{\varepsilon p} & 0 \leq j \leq \frac{p-1}{2} \text{ and } k=0\\
\dfrac{2\sqrt{\varepsilon p}}{\varepsilon p} & j=0 \text{ and } 1 \leq k \leq \frac{p-1}{2}\\
\dfrac{1}{\sqrt{\varepsilon p}}\left(e^{-2\pi i B(j,k)} + e^{2\pi i B(j,k)}\right) & \text{otherwise}
\end{array}
\right.
\end{split}
\end{equation}
and 
\begin{equation}
\label{eqn:xi-T-matrix}
T_{jk} = \left \lbrace \begin{array}{rl}
e^{2\pi i Q(j)}& \text{if } j=k\\
0 &  \text{otherwise.}
\end{array} \right. 
\end{equation}

Note that these models are not integral, since they can only be defined over the ring $\Z[1/p,\zeta_p]$. 


\begin{remark}
The $(p-1)/2$-dimensional cuspidal series representations $\pi_1$ and $\pi_2$ can be constructed similarly by restricting $\rho$ to the subspace $V^- \subseteq V$ of {\em odd} functions, those satisfying $f(-x) = -f(x)$. This subspace has basis

$$ b^-_0:=\delta_1 - \delta_{(p-1)},\ \  \cdots, \ \ b^-_{(p-3)/2} := \delta_{(p-1)/2} - \delta_{(p+1)/2}.$$ 

Let $S$ and $T$ denote the matrices of the Weil representations $\rho(\mathfrak{s})$ and $\rho(\mathfrak{t})$ respectively, restricted to the basis of odd delta functions. 
Then $S$ and $T$ are given by 
$$
S_{jk} =\dfrac{1}{\sqrt{\varepsilon p}}\left(e^{-2\pi i B(j+1,k+1)} - e^{2\pi i B(j+1,k+1)}\right),
$$
and 
$$
T_{jk} = \left \lbrace \begin{array}{rcl} 
e^{2\pi i Q(j+1)}& \text{if} & j=k\\
0 & & \text{otherwise.}
\end{array} \right.
$$

These formulas provide explicit models for $\pi_1$ and $\pi_2$ over $\Z[1/p, \zeta_p]$.

\end{remark}

Wang \cite{Yilong} provides an {\em integral} model for $\xi_2$ over the ring $\Z[\zeta_p]$, by using a basis consisting of circulant vectors. His construction can also be employed to provide an integral model for $\xi_1$, also defined over the ring $\Z[\zeta_p]$. Note that the existence of such integral models for $\xi_1, \xi_2$ was proved by Riese in \cite{Riese}, but no explicit integral models were given.

To recall Wang's construction, let $Q = Q_1, Q_2$ be one of the quadratic forms \eqref{eqn:Q1Q2}, corresponding to each Weil character $\xi_1, \xi_2$. Let $r=(p-1)/2$ and let $\theta_j= e^{2\pi i Q(j)}$ for $0 \leq j\leq r$, be the eigenvalues of the $T$-matrix \eqref{eqn:Weil-T_mat}. Note that $\theta_j \neq \theta_k$ for all $j\neq k$. Consequently, the Vandermonde matrix
\begin{equation}
\label{eqn:Vandermonde}
V_Q = \begin{bmatrix}
1&1&1 & 1 & \cdots & 1 \\
1&\theta_1&\theta_1^2 & \theta_1^3 & \cdots & \theta_1^{r} \\
1&\theta_2&\theta_2^2 & \theta_2^3 & \cdots & \theta_2^{r} 
 \\
1&\theta_3&\theta_3^2 & \theta_3^3 & \cdots & \theta_3^{r} 
\\
\vdots&\vdots & \vdots & \vdots  & \cdots & \vdots \\
1&\theta_{r}&\theta_{r}^2 & \theta_{r}^3 & \cdots &\theta_{r}^{r} 
\end{bmatrix}
\end{equation}
is invertible. Wang proves the following: 

\begin{theorem}[\cite{Yilong}, Thm. 1] 
\label{thm:WangsThm}
Suppose $p \geq 5$ is an odd prime. Let $S,T$ be the matrices \eqref{eqn:xi-S-matrix} and \eqref{eqn:xi-T-matrix}. Then the matrices $V_Q^{-1}S V_Q$ and $V_Q^{-1}TV_Q$ have entries in $\Z[\zeta_p]$. 
\end{theorem}

By Wang's theorem, setting 
$$
\xi_i(\mathfrak{s}) = V_{Q_i}^{-1}S V_{Q_i}, \quad \xi_i(\mathfrak{t}) = V_{Q_i}^{-1}T V_{Q_i}
$$
yields explicit integral models for $\xi_i$, $i=1,2$ over the ring $\ZZ[\zeta_p]$.

\begin{remark} In his proof, Wang shows that the representation $\xi_2$ is defined over $\Z[\zeta_p]$ for $ p \equiv 1 \mod 4$ and $\Z[\zeta_p,i]$ for $p \equiv 3 \mod 4$.  However, upon further inspection of Theorem 1, the proof readily generalizes to show integrality over $\Z[\zeta_p]$ for all odd primes and for both $\xi_1$ and $\xi_2$. 
\end{remark}

\begin{remark}
\label{remark:cuspidalreps}
Integral models for the cuspidal Weil characters (more precisely $\pi_1$) were first provided by Gilbert, Masbaum and van Wamelen \cite{PatrickM-2004} in the context of integral topological quantum field theories. We thank the referee for pointing out the correct reference to this work. More recently, Shaul Zemel has also provided similar integral models for $\pi_1$ and $\pi_2$ \cite{Shaul}. 
\end{remark}

\section{Minimal integral models and proof of the Main Theorem}
\label{Section:proofs}

In this section we prove that the integral models given by Theorem \ref{thm:WangsThm} are in fact defined over the ring of integers of $\mathbb{Q}(\sqrt{\varepsilon p})$, thus providing {\em minimal} integral models for $\xi_1$ and $\xi_2$. We do so by analyzing the action of the Galois group $\Gal(\Q(\zeta_p)/\Q)$ on the integral models. The key observation in this analysis is a surprising compatibility between the Galois actions on the Vandermonde matrix $V_Q$ and on the entries of the matrices of the Weil representations $\rho_Q$, for $Q = Q_1, Q_2$. 

Recall that for an odd prime $p$, the quadratic Gauss sums can be evaluated as follows:
\begin{equation}
\sum_{x \in \Z/p\Z} \zeta_p^{x^2} = \left \lbrace \begin{array}{ll} \sqrt{p} & \text{for } \ p \equiv 1 \mod 4\\
\\
\sqrt{-p} & \text{for } \ p \equiv 3 \mod 4,
\end{array} \right.  \nonumber
\end{equation}
and therefore the quadratic field $\Q(\sqrt{\varepsilon p})$ is always contained in $\Q(\zeta_p)$. By the fundamental theorem of Galois theory, this subfield must correspond to the subgroup $H\subseteq \Gal (\Q(\zeta_p)/\Q)$ of index 2, 
\begin{equation}
\label{eqn:GaloisDiagram}
\begin{tikzcd}[column sep=small]
\Q(\zeta_p) \arrow[d, no head,"r" ] &  & \{ e \} \arrow[d, no head,"r" ] \\
\Q(\sqrt{\varepsilon p}) \arrow[d, no head,"2" ] & \Longleftrightarrow & H = \langle \gamma^2 \rangle  \arrow[d, no head,"2" ] \\
\Q & & \langle \gamma \rangle 
\end{tikzcd}
\end{equation}
where in the diagram we chose a generator $ \langle \gamma \rangle = \Gal (\Q(\zeta_p)/\Q) \simeq \Z/(2r)\Z$ and $r = (p-1)/2$. The generator $\gamma$ can be written down as an automorphism $\gamma: \zeta_p \to \zeta_p^j$ where $\gcd(j,2r)=1$.  Then $\gamma^2: \zeta_p \to \zeta_p^{j^2}$ acts on the exponents of $\zeta_p$ by sending squares to squares and non-squares to non-squares. We can write the sums of square exponents and sums of nonsquares exponents in terms of the quadratic Gauss sum:
\begin{equation}
\sum_{ k \in (\F_p)^\times, \ k \ a \ square } \zeta_p^{k} = \dfrac{-1+\sqrt{\varepsilon p}}{2}  \quad  \text{and} \quad 
\sum_{ j \in (\F_p)^\times, \ j \ not \ a \ square } \zeta_p^{j} = \dfrac{-1-\sqrt{\varepsilon p}}{2} \nonumber
\end{equation}
and it is clear that they are always contained in the ring of integers $\Z[\frac{1}{2}(1+\sqrt{\varepsilon p})] \subset \Q(\sqrt{\varepsilon p})$.  

Let $S$ and $T$ be the matrices given in \eqref{eqn:xi-S-matrix} and \eqref{eqn:xi-T-matrix}, let $V_{Q}$ be the Vandermonde matrix \eqref{eqn:Vandermonde}, and let $T' = V_{Q}^{-1} T V_{Q}$. We first prove the integrality of  $T'$ over the quadratic ring:

\begin{theorem}\label{thm1}  The matrix entries of $T'$ lie in $ \Z[\frac{1}{2}(1+{\sqrt{\varepsilon p}})]$, the ring of integers of $\Q(\sqrt{\varepsilon p})$.
\end{theorem}
\begin{proof}
Note that $T'$ is the following circulant matrix 
\begin{equation}
\begin{split}
T'&=\begin{bmatrix}
0 & 0 & 0 & \cdots &0  &  -a_{r+1} \\
1 & 0 & 0 & \cdots & 0 & -a_{r} \\
0 & 1 & 0 & \cdots & 0 & -a_{r-1} \\
\vdots & \vdots & \vdots & \vdots & \vdots & \vdots \\
0 & 0 & 0 & \cdots & 1 & -a_1 \\
\end{bmatrix}, \nonumber
\end{split}
\end{equation}
whose characteristic polynomial $m(x) \in \Q(\zeta_p)[x]$ is given by 
\begin{equation}
\begin{split}
m(x) = x^{r+1} + a_1x^{r} + a_2x^{r-1} + \cdots + a_r x + a_{r+1}.
\end{split} \nonumber
\end{equation}
Thus to prove the Theorem it suffices to show that the coefficients of $m(x)$ are contained in $ \Z[\frac{1}{2}(1+{\sqrt{\varepsilon p}})]$. Since $T$ and $T'$ are conjugate matrices, we know that $m(x)$ splits as $(x-1)(x-\theta_1)(x-\theta_2)\cdots (x-\theta_{r})$ in $\Q(\zeta_p)$, where the $\theta_i$'s are the eigenvalues/diagonal entries of $T$. In the case $Q = Q_1$, each $\theta_j$ is of the form $\zeta_p^s$, where $s$ is a square mod $p$, while in the case $Q = Q_2$, each $\theta_i$ is of the form $\zeta_p^n$, where $n$ is not a square mod $p$. In each case, the set of roots of the polynomial $m(x)$ is permuted by the index 2 subgroup $H\subseteq \Gal (\Q(\zeta_p)/\Q)$ defined in \eqref{eqn:GaloisDiagram}. So $H$ fixes the coefficients of $m(x)$, since those are expressible in terms of elementary symmetric polynomials in the $\theta_j$'s.  Therefore the coefficients $a_i$ of $m(x)$ lie in the quadratic extension $\Q(\sqrt{\varepsilon p})$. In addition, since all the roots of unity $\zeta_p^j$ belong to the ring of integers $\Z[\zeta_p]$ and since the symmetric polynomials have integral coefficients, it must follow that the coefficients $a_j$ actually belong to the ring of integers of $\Q(\sqrt{\varepsilon p})$.  
\end{proof}

For any element $\alpha\in \Gal(\Q(\zeta_p)/\Q)$ and any matrix $M = (m_{jk})$ with entries  $m_{jk}\in \Q(\zeta_p)$, we define $\alpha(M) = ( \alpha (m_{jk}) )$ to be the matrix obtained from $M$ by applying the field automorphism $\alpha$ to each entry. More generally, if $\rho:G \rightarrow \GL_n(\Q(\zeta_p))$ is a model for a group representation, we denote by $\rho^{\alpha}$ the model for the group representation $\rho^{\alpha}(g) = \alpha(\rho(g))$, obtained by applying the field automorphism $\alpha$ to each matrix entry $\rho(g)$, $g\in G$. 

In particular, let $\tau=\gamma^2$ be the generator of the subgroup $H$ of the Galois group of $\Q(\zeta_p)$ defined in \eqref{eqn:GaloisDiagram}. Clearly we have $|\tau|= |H| = r$.  Consider the action of $\tau$ on the Vandermonde matrix $V_Q$. Since $\tau(\zeta_p^x) = \zeta_p^y$ where $x,y$ are either both squares or both non-squares mod $p$, and since $\tau$ fixes the number 1, it follows that $\tau(V_Q)$ is a matrix obtained by permuting the rows of $V_Q$. Let ${P}$ denote the permutation matrix corresponding to this permutation of the rows of $V_Q$: 
\begin{equation}
\label{eqn:P-perm}
\tau(V_{Q}) = {P} \cdot V_{Q}.
\end{equation}
Since $|\tau|=r$, the order of $P$ is also $r$. $P$ fixes the first row  and therefore $P^{-1}$ also fixes the first row. We can consider $P$ as a $r-1$ cycle.  We also note that $P^{-1}=P^t$, since the inverse of a permutation matrix is its transpose.

Interestingly, it turns out that the permutation matrix $P$ also gives the Galois action of $\tau$ on the Weil representation models for $\xi_1, \xi_2$ given by the matrices $S,T$ defined in  \eqref{eqn:xi-S-matrix} and  \eqref{eqn:xi-T-matrix}. This non-trivial compatibility between the Galois action on the Vandermonde matrix $V_Q$ and the Weil representation models is the key observation of this article:

\begin{theorem}
\label{thm2} 
Let $\xi_1, \xi_2: \SL_2(\F_p) \rightarrow \GL_{r+1}(\Q(\zeta_p))$ be the explicit models for the principal series Weil characters determined by $\xi_i(\mathfrak{s}) = S, \xi_i(\mathfrak{t}) = T$, where $S,T$ are given in \eqref{eqn:xi-S-matrix} and \eqref{eqn:xi-T-matrix}. Then for each $i = 1,2$ and all primes $p>2$, the permutation matrix $P$ given in \eqref{eqn:P-perm} satisfies 
$$
\xi_i^\tau(g) = P \,\xi_i(g) \, P^{-1}
$$
for all $g \in \SL_2(\F_p)$.
\end{theorem}

\begin{proof}
As is well-known, the character table entries for $\xi_i$ are defined over the quadratic field $\Q(\sqrt{\varepsilon p})$. Since this the fixed field of the subgroup $H = \langle \tau \rangle \subseteq \Gal(\Q(\zeta_p)/\Q)$, it follows that the character table entries of $\xi_i$ and $\xi_i^\tau$ are the same, so $\xi_{i}^\tau \cong \xi_{i}$. This implies that there exists a matrix $M\in \GL(V)$ (uniquely defined up to multiplication by a scalar) such that 
\begin{equation}
\label{eqn:M-iso}
\xi_{i}^\tau(g) = M^{-1} \cdot \xi_{i}(g) \cdot M 
\end{equation}
for all $g \in \SL_2(\F_p)$. We want to show that $M = \lambda P^{-1}$ where $\lambda \in \C$ and $P$ is the permutation matrix \eqref{eqn:P-perm}, determined by the relation $\tau(V_{Q}) = PV_Q$. First, recall that the matrix $T' = V_{Q}^{-1}T V_{Q}$ has entries in $\Q(\sqrt{\varepsilon p})$, by Theorem \ref{thm1}, and therefore $\tau(T') = T'$. It follows that 
$$
\tau(T)= \tau(V_{Q}) \cdot T' \cdot \tau(V_{Q})^{-1} = PV_{Q} \cdot T' \cdot V_{Q}^{-1} P^{-1} = PTP^{-1}.
$$
On the other hand, we know from \eqref{eqn:M-iso} that 
$$
\tau(T) = M^{-1}TM
$$
so that the matrix $T$ must commute with the matrix $P M$.  Since $T$ is a diagonal matrix whose entries $\theta_i$ are all {\em distinct}, it follows that $PM = D$ is also a diagonal matrix, and $M = P^{-1}D$. We want to show that $D = \lambda \cdot \mathds{1}$ in fact scalar. To show this, we also need to employ the $S$-matrix. Let 
\begin{equation}
D = \begin{bmatrix}
\lambda_1 & 0 & 0 &\cdots&  0\\
0& \lambda_2 & 0 &\cdots & 0\\
\vdots & \vdots & \vdots & \cdots & \vdots\\
0& 0 & 0 & \cdots & \lambda_{r} 
\end{bmatrix}. \nonumber
\end{equation}
Again recall from \eqref{eqn:M-iso} that $\tau(S) = M^{-1} \cdot S \cdot M$, so that 
$$
\tau(S) = M \cdot S \cdot  M^{-1} = P^{-1}D \cdot S \cdot D^{-1}P  
$$
Recall that $S$ has the form 
\begin{equation}
 S = \dfrac{\sqrt{\varepsilon p}}{\varepsilon p}\begin{bmatrix}
1 & 2 & 2 & \cdots & 2\\
1 & s_{11} & s_{12} & \cdots & s_{1r} \\
\vdots & \vdots & \vdots & \cdots & \vdots \\
1 & s_{r1} & s_{r2}& \cdots & s_{rr} \\
\end{bmatrix}, \nonumber
\end{equation}
for some entries $s_{jk}$. In particular, since the first row and first column of $S$ consists of elements of $\Q(\sqrt{\varepsilon p})$, the action of $\tau$ leaves them fixed. In addition, by definition the permutation matrix $P$ also fixes the first row and the first column, and so does $P^{-1} = P^t$.  Therefore we can write 
\begin{equation}
P\cdot \tau(S)\cdot P^{-1} =\dfrac{\sqrt{\varepsilon p}}{\varepsilon p} \begin{bmatrix}
1 & 2 & 2 & \cdots & 2\\
1 & x_{11} & x_{12} & \cdots & x_{1r} \\
\vdots & \vdots & \vdots & \cdots & \vdots \\
1 & x_{r 1} & x_{r2}& \cdots & x_{rr} \\
\end{bmatrix}, \nonumber
\end{equation}
for some entries $x_{jk}$'s. So 
\begin{align*}
D^{-1} \cdot P\cdot \tau(S)\cdot P^{-1} \cdot D &=\dfrac{\sqrt{\varepsilon p}}{\varepsilon p} \begin{bmatrix} 1 & 2 \lambda_1 \lambda_2^{-1} & 2\lambda_1 \lambda_3^{-1} & \cdots & 2\lambda_1 \lambda_{r}^{-1}\\
\lambda_2\lambda_1^{-1} & * & * & \cdots & * \\
\vdots & \vdots & \vdots & \cdots & \vdots \\
\lambda_{r}\lambda_1^{-1} & * & * & \cdots & * \\
\end{bmatrix}  \\
&= \dfrac{\sqrt{\varepsilon p}}{\varepsilon p}\begin{bmatrix}
1 & 2 & 2 & \cdots & 2\\
1 & s_{11} & s_{12} & \cdots & s_{1r} \\
\vdots & \vdots & \vdots & \cdots & \vdots \\
1 & s_{r1} & s_{r2}& \cdots & s_{rr} \\
\end{bmatrix} = S.
\end{align*}
Equating the first column and row of these matrices we conclude that $\lambda_j = \lambda_k = \lambda$ are all equal, and thus $M = \lambda P^{-1}$. 
\end{proof}

 Let now $S' := V_{Q}^{-1} \cdot S \cdot V_{Q}$. Theorem \ref{thm2} enables us to prove the following theorem. 

\begin{theorem} \label{thm3}  $S'$ has matrix entries in $ \Z[\frac{1}{2}(1+{\sqrt{\varepsilon p}})]$, the ring of integers of $\Q(\sqrt{\varepsilon p})$.
\end{theorem}
\begin{proof}
Again let $\tau$ be the generator of the subgroup $H \subseteq \Gal(\Q(\zeta_p)/\Q)$ defined in \eqref{eqn:GaloisDiagram}. We want to show that $\tau(S') = S'$. By Theorem \ref{thm3}, we have: 

\begin{equation}
\begin{split}
\tau(S')&=\tau(V_{Q}^{-1} \cdot S \cdot V_{Q} ) \\
&= \tau(V_{Q}^{-1})\cdot \tau(S) \cdot \tau(V_{Q})\\
&=V_{Q}^{-1}\cdot P^{-1}\cdot P \cdot S \cdot P^{-1} \cdot P \cdot  V_{Q}\\
&=V_{Q}^{-1}\cdot S \cdot V_{Q}\\
&=S'.
\end{split} \nonumber
\end{equation}
Since $\tau$ fixes $S'$, the entries of $S'$ lie in the quadratic extension $\Q(\sqrt{\varepsilon p})$. But we also know from Thm. \ref{thm:WangsThm} that the entries of $S'$ belong to the ring of algebraic integers $\Z[\zeta_p]$, therefore they must lie in the ring of integers of $\Q(\sqrt{\varepsilon p})$.  
\end{proof}

By Theorems \ref{thm1} and \ref{thm3}, setting   
$$
\xi_i(\mathfrak{s}) = S', \xi_i(\mathfrak{t}) = T'
$$
gives integral models for the principal series Weil characters over the ring of integers of $\Q(\sqrt{\varepsilon p})$, therefore giving minimal integral models for these characters. 

\begin{remark}
As pointed out in Remark \ref{remark:cuspidalreps}, there are integral models available for the cuspidal Weil characters as well. It does not appear however that these models are defined over smaller rings. In particular, minimal integral models for cuspidal Weil characters remain unknown. 
\end{remark}

\begin{remark}
The referee pointed out to the authors that Theorem \ref{thm2} can be obtained as a special case of the Galois symmetry for modular tensor categories (MTCs). More precisely, according to \cite{Dong-Lin-Ng}, Theorem II, if $\rho:\SL_2(\ZZ) \rightarrow \GL_m(\mathbb{C})$ is the standard model for the modular representation of a MTC, then for any $\sigma \in \mathrm{Gal}(\Q(\zeta_p)/\Q)$ the action of $\sigma^2$ on the matrix entries of $\rho$ is given by conjugation by a signed permutation matrix. Although the principal series Weil characters $\xi_1, \xi_2$ cannot be obtained directly as modular representations (see for example \cite{reconstruction}, Prop. 3.22), the full Weil character $\xi_i\oplus \pi_i$ is the modular representation of the pointed MTC whose fusion algebra is $\mathbb{C}[\ZZ/p\ZZ]$ (with the appropriate choice of quadratic form). Theorem \ref{thm2} can then be obtained with some work from the above Galois symmetry result for a general MTC. 
\end{remark}

\section{Examples}
\label{sec:examples}

We now show the computations giving the explicit minimal integral models for $\xi_1$ and $\xi_2$ for the primes $p=7$ and $p=13$. We used SageMath \cite{sagemath} and MATLAB \cite{MATLAB:2015} to make and verify the computations.

\begin{example}
Let $p=7$. For the equivalent representation of $\xi_1$, let $c=1$. So,  $Q_1(x)=x^2/7$,  $B_1(x,y)=2xy/7$ and the Gauss sum is given by 
$$ \sum_{x \in \Z/7\Z} \zeta_7^{x^2} = 2 \zeta_{7}^{4} + 2 \zeta_{7}^{2} + 2 \zeta_{7} + 1 = \sqrt{-7}.$$

Then we have 
$$ S = \dfrac{-\sqrt{7}}{7} \begin{bmatrix}
1 & 2 & 2 & 2 \\
1 & \zeta_{7}^{5} + \zeta_{7}^{2} & \zeta_{7}^{4} + \zeta_{7}^{3} & \zeta_{7}^{6} + \zeta_{7} \\
1 & \zeta_{7}^{4} + \zeta_{7}^{3} & \zeta_{7}^{6} + \zeta_{7}  & \zeta_{7}^{5} + \zeta_{7}^{2} \\
1 & \zeta_{7}^{6} + \zeta_{7}  & \zeta_{7}^{5} + \zeta_{7}^{2} & \zeta_{7}^{4} + \zeta_{7}^{3}
\end{bmatrix},$$
$$T = \begin{bmatrix}
1 & 0 & 0 & 0 \\
0 & \zeta_{7} & 0 & 0 \\
0 & 0 & \zeta_{7}^{4} & 0 \\
0 & 0 & 0 & \zeta_{7}^{2}
\end{bmatrix},$$
$$V_Q=\begin{bmatrix}
1 & 1 & 1 & 1 \\
1 & \zeta_{7} & \zeta_{7}^{2} & \zeta_{7}^{3} \\
1 & \zeta_{7}^{4} & \zeta_{7} & \zeta_{7}^{5} \\
1 & \zeta_{7}^{2} & \zeta_{7}^{4} & \zeta_{7}
\end{bmatrix},$$
$$S'=V_Q^{-1} S V_Q = \begin{bmatrix}
-1 & \frac{1}{2}(1-\sqrt{-7}) & 0 & 0 \\
\frac{1}{2}(-1-\sqrt{-7}) & 1 & 0 & 0 \\
\frac{1}{2}(1-\sqrt{-7})& \frac{1}{2}(1+\sqrt{-7}) & 0 & -1 \\
1 & -1 & 1 & 0
\end{bmatrix},$$
and
$$T'=V_Q^{-1} T V_Q = \begin{bmatrix}
0 & 0 & 0 & -1 \\
1 & 0 & 0 & \frac{1}{2}(1-\sqrt{-7}) \\
0 & 1 & 0 & 1 \\
0 & 0 & 1 & \frac{1}{2}(1+\sqrt{-7})
\end{bmatrix}.$$

For the equivalent representation of $\xi_2$, we choose $c=3$.  So, $Q_2(x)=3x^2/7$,  $B_2(x,y)=6xy/7$ and the Gauss sum is given by 
$$ \sum_{x \in \Z/7\Z} \zeta_7^{3x^2} = -2 \zeta_{7}^{4} - 2 \zeta_{7}^{2} - 2 \zeta_{7} - 1 = -\sqrt{-7}.$$
Then we have 
$$ S = \dfrac{\sqrt{-7}}{7} \begin{bmatrix}
1 & 2 & 2 & 2 \\
1 & \zeta_{7}^{6} + \zeta_{7} & \zeta_{7}^{5} + \zeta_{7}^{2} & \zeta_{7}^{4} + \zeta_{7}^{3} \\
1 & \zeta_{7}^{5} + \zeta_{7}^{2} & \zeta_{7}^{4} + \zeta_{7}^{3} & \zeta_{7}^{6} + \zeta_{7} \\
1 & \zeta_{7}^{4} + \zeta_{7}^{3} & \zeta_{7}^{6} + \zeta_{7}^{1} & \zeta_{7}^{5} + \zeta_{7}^{2}
\end{bmatrix},$$
$$T = \begin{bmatrix}
1 & 0 & 0 & 0 \\
0 & \zeta_{7}^{3} & 0 & 0 \\
0 & 0 & \zeta_{7}^{5} & 0 \\
0 & 0 & 0 & \zeta_{7}^{6}
\end{bmatrix},$$
$$V_Q=\begin{bmatrix}
1 & 1 & 1 & 1 \\
1 & \zeta_{7}^{3} & \zeta_{7}^{6} & \zeta_{7}^{2} \\
1 & \zeta_{7}^{5} & \zeta_{7}^{3} & \zeta_{7} \\
1 & \zeta_{7}^{6} & \zeta_{7}^{5} & \zeta_{7}^{4}
\end{bmatrix},$$
$$S'=V_Q^{-1} S V_Q = \begin{bmatrix}
-1 & \frac{1}{2}(1+\sqrt{-7}) & 0 & 0 \\
\frac{1}{2}(-1+\sqrt{-7}) & 1 & 0 & 0 \\
\frac{1}{2}(1+\sqrt{-7}) & \frac{1}{2}(1-\sqrt{-7}) & 0 & -1 \\
1 & -1 & 1 & 0
\end{bmatrix},$$
and
$$T'=V_Q^{-1} T V_Q = \begin{bmatrix}
0 & 0 & 0 & -1 \\
1 & 0 & 0 & \frac{1}{2}(1+\sqrt{7})\\
0 & 1 & 0 & 1 \\
0 & 0 & 1 & \frac{1}{2}(1-\sqrt{7})
\end{bmatrix}.$$
\end{example}

\begin{example} Let $p=13$. For the equivalent representation of $\xi_1$, let $c=1$. So,  $Q_1(x)=x^2/13$,  $B_1(x,y)=2xy/13$ and the Gauss sum is given by 
$$ \sum_{x \in \Z/13\Z} \zeta_{13}^{x^2} = -2 \zeta_{13}^{11} - 2 \zeta_{13}^{8} - 2 \zeta_{13}^{7} - 2 \zeta_{13}^{6} - 2 \zeta_{13}^{5} - 2 \zeta_{13}^{2} - 1=\sqrt{13}.$$

Then we have 
$$ S = \dfrac{\sqrt{13}}{13} \begin{bmatrix}
1 & 2 & 2 & 2 & 2 & 2 & 2 \\
1 & \zeta_{13}^{11} + \zeta_{13}^{2} & \zeta_{13}^{9} + \zeta_{13}^{4} & \zeta_{13}^{7} + \zeta_{13}^{6} & \zeta_{13}^{8} + \zeta_{13}^{5} & \zeta_{13}^{10} + \zeta_{13}^{3} & \zeta_{13}^{12} + \zeta_{13} \\
1 & \zeta_{13}^{9} + \zeta_{13}^{4} & \zeta_{13}^{8} + \zeta_{13}^{5} & \zeta_{13}^{12} + \zeta_{13} & \zeta_{13}^{10} + \zeta_{13}^{3} & \zeta_{13}^{7} + \zeta_{13}^{6} & \zeta_{13}^{11} + \zeta_{13}^{2} \\
1 & \zeta_{13}^{7} + \zeta_{13}^{6} & \zeta_{13}^{12} + \zeta_{13} & \zeta_{13}^{8} + \zeta_{13}^{5} & \zeta_{13}^{11} + \zeta_{13}^{2} & \zeta_{13}^{9} + \zeta_{13}^{4} & \zeta_{13}^{10} + \zeta_{13}^{3} \\
1 & \zeta_{13}^{8} + \zeta_{13}^{5} & \zeta_{13}^{10} + \zeta_{13}^{3} & \zeta_{13}^{11} + \zeta_{13}^{2} & \zeta_{13}^{7} + \zeta_{13}^{6} &\zeta_{13}^{12} + \zeta_{13} & \zeta_{13}^{9} + \zeta_{13}^{4} \\
1 & \zeta_{13}^{10} + \zeta_{13}^{3} & \zeta_{13}^{7} + \zeta_{13}^{6} & \zeta_{13}^{9} + \zeta_{13}^{4} & \zeta_{13}^{12} + \zeta_{13} & \zeta_{13}^{11} + \zeta_{13}^{2} & \zeta_{13}^{8} + \zeta_{13}^{5} \\
1 & \zeta_{13}^{12} + \zeta_{13} & \zeta_{13}^{11} + \zeta_{13}^{2} & \zeta_{13}^{10} + \zeta_{13}^{3} & \zeta_{13}^{9} + \zeta_{13}^{4} & \zeta_{13}^{8} + \zeta_{13}^{5} & \zeta_{13}^{7} + \zeta_{13}^{6}
\end{bmatrix},$$
$$T = \begin{bmatrix}
1 & 0 & 0 & 0 & 0 & 0 & 0 \\
0 & \zeta_{13} & 0 & 0 & 0 & 0 & 0 \\
0 & 0 & \zeta_{13}^{4} & 0 & 0 & 0 & 0 \\
0 & 0 & 0 & \zeta_{13}^{9} & 0 & 0 & 0 \\
0 & 0 & 0 & 0 & \zeta_{13}^{3} & 0 & 0 \\
0 & 0 & 0 & 0 & 0 & \zeta_{13}^{12} & 0 \\
0 & 0 & 0 & 0 & 0 & 0 & \zeta_{13}^{10}
\end{bmatrix},$$
$$V_Q=\begin{bmatrix}
1 & 1 & 1 & 1 & 1 & 1 & 1 \\
1 & \zeta_{13} & \zeta_{13}^{2} & \zeta_{13}^{3} & \zeta_{13}^{4} & \zeta_{13}^{5} & \zeta_{13}^{6} \\
1 & \zeta_{13}^{4} & \zeta_{13}^{8} & \zeta_{13}^{12} & \zeta_{13}^{3} & \zeta_{13}^{7} & \zeta_{13}^{11} \\
1 & \zeta_{13}^{9} & \zeta_{13}^{5} & \zeta_{13} & \zeta_{13}^{10} & \zeta_{13}^{6} & \zeta_{13}^{2} \\
1 & \zeta_{13}^{3} & \zeta_{13}^{6} & \zeta_{13}^{9} & \zeta_{13}^{12} & \zeta_{13}^{2} & \zeta_{13}^{5} \\
1 & \zeta_{13}^{12} & \zeta_{13}^{11} & \zeta_{13}^{10} & \zeta_{13}^{9} & \zeta_{13}^{8} & \zeta_{13}^{7} \\
1 & \zeta_{13}^{10} & \zeta_{13}^{7} & \zeta_{13}^{4} & \zeta_{13} & \zeta_{13}^{11} & \zeta_{13}^{8}
\end{bmatrix},$$
$$S'=V_Q^{-1} S V_Q = \begin{bmatrix}
\frac{1}{2}(3+\sqrt{13}) & \frac{1}{2}(1+\sqrt{13}) & 0 & 0 & 0 & 0 & 0 \\
-\frac{1}{2}(5+\sqrt{13}) & -\frac{1}{2}(3+\sqrt{13}) & 0 & 0 & 0 & 0 & 0 \\
3+\sqrt{13} & \frac{1}{2}(5+\sqrt{13})& 0 & 0 & 0 & 0 & -1 \\
-4-\sqrt{13} & -\frac{1}{2}(5+\sqrt{13}) & 0 & 0 & 1 & 0 & 0 \\
3+\sqrt{13} & \frac{1}{2}(3+\sqrt{13}) & 0 & 1 & 0 & 0 & 0 \\
-\frac{1}{2}(5+\sqrt{13}) & -\frac{1}{2}(1+\sqrt{13}) & 0 & 0 & 0 & -1 & 0 \\
\frac{1}{2}(3+\sqrt{13}) & 1 & -1 & 0 & 0 & 0 & 0
\end{bmatrix},$$
and
$$T'=V_Q^{-1} T V_Q = \begin{bmatrix}
0 & 0 & 0 & 0 & 0 & 0 & 1 \\
1 & 0 & 0 & 0 & 0 & 0 &-\frac{1}{2}(1+\sqrt{13}) \\
0 & 1 & 0 & 0 & 0 & 0 & \frac{1}{2}(3 +\sqrt{13}) \\
0 & 0 & 1 & 0 & 0 & 0 & -\frac{1}{2}(5+\sqrt{13}) \\
0 & 0 & 0 & 1 & 0 & 0 & \frac{1}{2}(5 +\sqrt{13}) \\
0 & 0 & 0 & 0 & 1 & 0 & -\frac{1}{2}(3 +\sqrt{13}) \\
0 & 0 & 0 & 0 & 0 & 1 & \frac{1}{2}(1+\sqrt{13})
\end{bmatrix}.$$

For the equivalent representation of $\xi_2$, let $c=2$. So,  $Q_2(x)=2x^2/13$,  $B_2(x,y)=4xy/13$ and the Gauss sum is given by 
$$ \sum_{x \in \Z/13\Z} \zeta_{13}^{x^2} = 2 \zeta_{13}^{11} + 2 \zeta_{13}^{8} + 2 \zeta_{13}^{7} + 2 \zeta_{13}^{6} + 2 \zeta_{13}^{5} + 2 \zeta_{13}^{2} + 1=-\sqrt{13}.$$

Then we have 
$$ S = \dfrac{-\sqrt{13}}{13} \begin{bmatrix}
1 & 2 & 2 & 2 & 2 & 2 & 2 \\
1 & \zeta_{13}^{9} + \zeta_{13}^{4} & \zeta_{13}^{8} + \zeta_{13}^{5} & \zeta_{13}^{12}+\zeta_{13} & \zeta_{13}^{10} + \zeta_{13}^{3} & \zeta_{13}^{7} + \zeta_{13}^{6} & \zeta_{13}^{11} + \zeta_{13}^{2} \\
1 & \zeta_{13}^{8} + \zeta_{13}^{5} & \zeta_{13}^{10} + \zeta_{13}^{3} & \zeta_{13}^{11} + \zeta_{13}^{2} & \zeta_{13}^{7} + \zeta_{13}^{6} & \zeta_{13}^{12}+\zeta_{13} & \zeta_{13}^{9} + \zeta_{13}^{4} \\
1 & \zeta_{13}^{12}+\zeta_{13} & \zeta_{13}^{11} + \zeta_{13}^{2} & \zeta_{13}^{10} + \zeta_{13}^{3} & \zeta_{13}^{9} + \zeta_{13}^{4} & \zeta_{13}^{8} + \zeta_{13}^{5} & \zeta_{13}^{7} + \zeta_{13}^{6} \\
1 & \zeta_{13}^{10} + \zeta_{13}^{3} & \zeta_{13}^{7} + \zeta_{13}^{6} & \zeta_{13}^{9} + \zeta_{13}^{4} & \zeta_{13}^{12}+\zeta_{13} & \zeta_{13}^{11} + \zeta_{13}^{2} & \zeta_{13}^{8} + \zeta_{13}^{5} \\
1 & \zeta_{13}^{7} + \zeta_{13}^{6} & \zeta_{13}^{12}+\zeta_{13} & \zeta_{13}^{8} + \zeta_{13}^{5} & \zeta_{13}^{11} + \zeta_{13}^{2} & \zeta_{13}^{9} + \zeta_{13}^{4} & \zeta_{13}^{10} + \zeta_{13}^{3} \\
1 & \zeta_{13}^{11} + \zeta_{13}^{2} & \zeta_{13}^{9} + \zeta_{13}^{4} & \zeta_{13}^{7} + \zeta_{13}^{6} & \zeta_{13}^{8} + \zeta_{13}^{5} & \zeta_{13}^{10} + \zeta_{13}^{3} & \zeta_{13}^{12}+\zeta_{13}
\end{bmatrix},$$
$$T = \begin{bmatrix}
1 & 0 & 0 & 0 & 0 & 0 & 0 \\
0 & \zeta_{13}^{2} & 0 & 0 & 0 & 0 & 0 \\
0 & 0 & \zeta_{13}^{8} & 0 & 0 & 0 & 0 \\
0 & 0 & 0 & \zeta_{13}^{5} & 0 & 0 & 0 \\
0 & 0 & 0 & 0 & \zeta_{13}^{6} & 0 & 0 \\
0 & 0 & 0 & 0 & 0 & \zeta_{13}^{11} & 0 \\
0 & 0 & 0 & 0 & 0 & 0 & \zeta_{13}^{7}
\end{bmatrix},$$
$$V_Q=\begin{bmatrix}
1 & 1 & 1 & 1 & 1 & 1 & 1 \\
1 & \zeta_{13}^{2} & \zeta_{13}^{4} & \zeta_{13}^{6} & \zeta_{13}^{8} & \zeta_{13}^{10} & \zeta_{13}^{12} \\
1 & \zeta_{13}^{8} & \zeta_{13}^{3} & \zeta_{13}^{11} & \zeta_{13}^{6} & \zeta_{13} & \zeta_{13}^{9} \\
1 & \zeta_{13}^{5} & \zeta_{13}^{10} & \zeta_{13}^{2} & \zeta_{13}^{7} &\zeta_{13}^{12}  & \zeta_{13}^{4} \\
1 & \zeta_{13}^{6} & \zeta_{13}^{12}  & \zeta_{13}^{5} & \zeta_{13}^{11} & \zeta_{13}^{4} & \zeta_{13}^{10} \\
1 & \zeta_{13}^{11} & \zeta_{13}^{9} & \zeta_{13}^{7} & \zeta_{13}^{5} & \zeta_{13}^{3} & \zeta_{13} \\
1 & \zeta_{13}^{7} & \zeta_{13} & \zeta_{13}^{8} & \zeta_{13}^{2} & \zeta_{13}^{9} & \zeta_{13}^{3}
\end{bmatrix},$$
$$S'=V_Q^{-1} S V_Q = \begin{bmatrix}
\frac{1}{2}(3-\sqrt{13}) & \frac{1}{2}(1-\sqrt{13}) & 0 & 0 & 0 & 0 & 0 \\
\frac{1}{2}(-5+\sqrt{13}) & \frac{1}{2}(-3+\sqrt{13})& 0 & 0 & 0 & 0 & 0 \\
3-\sqrt{13} & \frac{1}{2}(5-\sqrt{13}) & 0 & 0 & 0 & 0 & -1 \\
-4+\sqrt{13}) & \frac{1}{2}(-5+\sqrt{13}) & 0 & 0 & 1 & 0 & 0 \\
3-\sqrt{13} &\frac{1}{2}(3-\sqrt{13}) & 0 & 1 & 0 & 0 & 0 \\
\frac{1}{2}(-5+\sqrt{13}) & \frac{1}{2}(-1+\sqrt{13}) & 0 & 0 & 0 & -1 & 0 \\
\frac{1}{2}(3-\sqrt{13}) & 1 & -1 & 0 & 0 & 0 & 0
\end{bmatrix},$$
and
$$T'=V_Q^{-1} T V_Q = \begin{bmatrix}
0 & 0 & 0 & 0 & 0 & 0 & 1 \\
1 & 0 & 0 & 0 & 0 & 0 & \frac{1}{2}(-1+\sqrt{13}) \\
0 & 1 & 0 & 0 & 0 & 0 &\frac{1}{2}(3-\sqrt{13}) \\
0 & 0 & 1 & 0 & 0 & 0 & \frac{1}{2}(-5+\sqrt{13}) \\
0 & 0 & 0 & 1 & 0 & 0 & \frac{1}{2}(5-\sqrt{13}) \\
0 & 0 & 0 & 0 & 1 & 0 & \frac{1}{2}(-3+\sqrt{13})\\
0 & 0 & 0 & 0 & 0 & 1 & \frac{1}{2}(1-\sqrt{13})
\end{bmatrix}.$$

\end{example}


\bibliographystyle{amsplain}
\bibliography{IntegralWeilCharacters}

\end{document}